\documentclass[12pt]{amsart}
\usepackage{amsfonts}
\usepackage{amssymb}
\usepackage{amsmath}
\usepackage{amsthm}

\newtheorem{theorem}{Theorem}
\newtheorem{corollary}{Corollary}
\newtheorem{lemma}{Lemma}

\theoremstyle{definition}
\newtheorem{definition}{Definition}

\begin{document}

\title[]{Holomorphic Motions, Fatou Linearization, and Quasiconformal Rigidity for Parabolic
Germs}

\author{Yunping Jiang}

\date {July 20, 2007}

\address{Department of Mathematics\\
Queens College of the City University of New York\\
Flushing, NY 11367-1597\\
and\\
Department of Mathematics\\
Graduate School of the City University of New York\\
365 Fifth Avenue, New York, NY 10016\\
and\\
Current Address: School of Mathematical Sciences\\
Fudan University, Shanghai 200433, China}
\email[]{yunping.jiang@qc.cuny.edu}



\thanks{The research is partially supported by PSC-CUNY awards}

\begin{abstract}
By applying holomorphic motions, we prove that a parabolic germ is
quasiconformally rigid, that is, any two topologically conjugate
parabolic germs are quasiconformally conjugate and the conjugacy
can be chosen to be more and more near conformal as long as we
consider these germs defined on smaller and smaller neighborhoods.
Before proving this theorem, we use the idea of holomorphic
motions to give a conceptual proof of the Fatou linearization
theorem. As a by-product, we also prove that any finite number of
analytic germs at different points in the Riemann sphere can be
extended to a quasiconformal homeomorphism which can be more and
more near conformal as as long as we consider these germs defined
on smaller and smaller neighborhoods of these points.
\end{abstract}

\maketitle

\section{Introduction}

One of the fundamental theorems in complex dynamical systems is a
theorem called the Fatou linearization theorem. This theorem
provides a topological and dynamical structure of a parabolic
germ. A parabolic germ $f$ is an analytic function defined in a
neighborhood of a point $z_{0}$ in the complex plane ${\mathbb C}$
such that it fixes $z_{0}$ and some power $(f'(z_{0}))^{q}$ of the
derivative $f'(z_{0})$ of f at $z_{0}$ is $1$. Thus we can write
$f(z)$ into the following form:
$$
f(z) = z_{0} + \lambda (z-z_{0}) +a_{2}(z-z_{0})^{2} +\cdots
+a_{n}(z-z_{0})^{n}+\cdots, \quad z\in U
$$
where $U$ is a neighborhood of $z_{0}$ and $\lambda = e^{2\pi p i
\over q}$, where $p$ and $q$ are two integers relatively prime.
The number $\lambda$ is called the multiplier of $f$. Two
parabolic germs $f$ and $g$ at two points $z_{0}$ and $z_{1}$ are
said to be topologically conjugate if there is a homeomorphism $h$
from a neighborhood of $z_{0}$ onto a neighborhood of $z_{1}$ such
that
$$
h\circ f=g\circ h.
$$
If $h$ is a $K$-quasiconformal homeomorphism, then we say that $f$
and $g$ are $K$-quasiconformally conjugate.

By a linear conjugacy $\phi(z)= z-z_{0}$, we may assume that
$z_{0}=0$. So we only consider parabolic germs at $0$,
$$
f(z) = \lambda z +a_{2}z^{2} +\cdots +a_{n}z^{n}+\cdots, \quad
z\in U.
$$

Assume we are given a parabolic germ $f$ at $0$ whose multiplier
$\lambda =e^{2\pi i p/q}$, $(p,q)=1$. Then
$$
f^{q}(z) =z +az^{n+1} + o(z^{n+1}), \quad n\geq 1.
$$
If $a\not=0$, then $n+1$ is called the multiplicity of $f$. Here
$n=kq$ is a multiplier of $q$. The Leau-Fatou flower theorem says
that the local topological and dynamical picture of $f$ around $0$
can be described as follows: There are $n$ petals pairwise
tangential at $0$ such that each petal is mapped into the
$(kp)^{th}$-petal counting counter-clockwise from this petal.
These petals are called attracting petals. At the same time, there
are $n$-repelling petals, that is, there are $n$ other petals also
pairwise tangential at $0$ and the inverse $f^{-1}$ maps each
petal into the $(kp)^{th}$-petal counting counter-clockwise from
this petal. Thus $f^{q}$ maps every attracting petal into itself
and $f^{-q}$ maps every repelling petal into itself. Furthermore,
the Fatou linearization theorem says that the map
$$
f^{q} (z) : {\mathcal P}\to {\mathcal P}
$$
from any attracting petal ${\mathcal P}$ into itself is conjugate
to $G(w) =w+1$ from a right half-plane into itself by a conformal
map.

The union of all attracting petals and repelling petals forms a
neighborhood of $0$. If two parabolic germs are topologically
conjugate, then they have the same Leau-Fatou flowers in any
neighborhood of $0$. A parabolic germ is quasiconformal rigidity
as follows.

\vspace*{10pt}
\begin{theorem}~\label{qcr}
Suppose $f$ and $g$ are two parabolic germs at $0$ and suppose $f$
and $g$ are topologically conjugate. Then for every $\epsilon
>0$ there are neighborhoods $U_{\epsilon}$ and $V_{\epsilon}$
about $0$ such that $f|U_{\epsilon}$ and
$g|V_{\epsilon}$ are $(1+\epsilon)$-quasiconformally conjugate.
\end{theorem}

\vspace*{10pt} We give a proof of this theorem. The method in our
proof is again to use holomorphic motions as we did
in~\cite{Jiang1,Jiang2}, where we involved holomorphic motions
into some new proofs of famous K\"onig's theorem and B\"ottcher's
theorem, which provide normal forms of attracting and
super-attracting germs. We continue this idea for parabolic germs
in this paper.

The original of this theorem is not very clear to me. It seems
that Ecalle had an intensive study of parabolic germs and may have
already had certain result toward this theorem. Camanche probably
did some studies on the dynamics of parabolic germs and already
realized some geometric properties of the conjugacy between two
topologically conjugate parabolic germs. A written proof was given
by McMullen~\cite[Theorem 8.1]{McMullenBook} for a special case.
In the proof, he used Ahlfors-Weill's extension theorem which says
that a conformal mapping of the unit disk can be extended to a
quasiconformal homeomorphism as long as the hyperbolic norm of the
Schwarzian derivative of this conformal mapping is less than $2$.
Moreover, the quasiconformal dilatation can tend to $1$ as the
hyperbolic norm goes to $0$. Furthermore, Cui proposes a proof by
using Ecalle cylinders and the major inequality in Teichm\"uller
theory in~\cite{Cui}.

The paper is organized as follows. Since our proof involves
holomorphic motions, we would like to introduce this interesting
topic in \S1. In \S2, we will continue the idea in
~\cite{Jiang1,Jiang2} to show a conceptual proof of the Fatou
linearization theorem by involving holomorphic motions. In \S3, we
give prove Theorem~\ref{qcr}. As a by-product, we prove
Theorem~\ref{conn} and Corllary~\ref{connmore} which may be
considered in some sense a generalization of Alhfors-Weill's
extension theorem.

\section{Holomorphic Motions and Quasiconformal Maps}

In the study of complex analysis, the measurable Riemann mapping
theorem plays an important role. A measurable function $\mu$ on
$\hat{\mathbb C}$ is called a Beltrami coefficient if its
$L^{\infty}$-norm
$$
k=\|\mu\|_{\infty} < 1.
$$
The corresponding equation
$$
H_{\overline{z}}=\mu H_{z}
$$
is called the Beltrami equation. The measurable Riemann mapping
theorem says that the Beltrami equation has a solution $H$ which
is a quasiconformal homeomorphism of $\hat{\mathbb C}$ whose
quasiconformal dilatation is less than or equal to
$K=(1+k)/(1-k)$. It is called a $K$-quasiconformal homeomorphism.

The study of the measurable Riemann mapping theorem has a long
history since Gauss considered in the 1820's the connection with
the problem of finding isothermal coordinates for a given surface.
As early as 1938, Morrey~\cite{Morrey} systematically studied
homeomorphic $L^{2}$-solutions of the Beltrami equation. But it
took almost twenty years until in 1957 Bers~\cite{Bers} observed
that these solutions are quasiconformal (refer to~\cite[pp.
24]{Lehto}). Finally the existence of a solution to the Beltrami
equation under the most general possible circumstance, namely, for
measurable $\mu$ with $\|\mu\|_{\infty}<1$, was shown by
Bojarski~\cite{Bojarski}. In this generality the existence theorem
is sometimes called the measurable Riemann mapping theorem.

If one only considers a normalized solution to the Beltrami
equation (a solution that fixes $0$, $1$, and $\infty$), then $H$
is unique, and the solution is denoted by $H^{\mu}$. The solution
$H^{\mu}$ is expressed as a power series made up of compositions
of singular integral operators applied to the Beltrami equation on
the Riemann sphere. In this expression, if one considers $\mu$ as
a variable, then the solution $H^{\mu}$ depends on $\mu$
analytically. This analytic dependence was emphasized by Ahlfors
and Bers in their 1960 paper~\cite{AhlforsBers} and is essential
in determining a complex structure for Teichm\"uller space (refer
to~\cite{AhlforsBook,Lehto,Li,Nag}). Note that when $\mu\equiv 0$,
$H^{0}$ is the identity map. A $1$-quasiconformal homeomorphism is
conformal.

Twenty years later, due to the development of complex dynamics,
this analytic dependence presents an even more interesting
phenomenon called holomorphic motions as follows. Let
$$
\Delta_{r}=\{ c\in {\mathbb C}\;\; |\;\; |c|<r\}
$$
denote the disk of radius $0<r<1$ centered at $0$. We use $\Delta$
to mean the unit disk. Given a measurable function $\mu$ on
$\hat{\mathbb C}$ with $\|\mu\|_{\infty}=1$, we have a family of
Beltrami coefficients $c\mu$ for $c\in \Delta$ and the family of
normalized solutions $H^{c\mu}$. Note that $H^{c\mu}$ is a
$(1+|c|)/(1-|c|)$-quasiconformal homeomorphism. Moreover,
$H^{c\mu}$ is a family which is holomorphic on $c$. Consider a
subset $E$ of $\hat{\mathbb C}$ and its image $E_{c}=H^{c\mu}(E)$.
One can see that $E_{c}$ moves holomorphically in $\hat{\mathbb
C}$ when $c$ moves in $\Delta$. That is, for any point $z\in E$,
$z(c)=H^{c\mu}(z)$ traces a holomorphic path starting from $z$ as
$c$ moves in the unit disk. Surprisingly, the converse of the
above fact is also true too. This starts from the famous
$\lambda$-lemma of Ma\~n\'e, Sad, and
Sullivan~\cite{ManeSadSullivan} in complex dynamical systems. Let
us start to understand this fact by first defining holomorphic
motions.

\medskip
\begin{definition}[Holomorphic Motions]~\label{hm} Let $E$ be a
subset of $\hat{\mathbb C}$. Let
$$
h (c, z) : \Delta\times E\to \hat{\mathbb C}
$$
be a map. Then $h$ is called a holomorphic motion of $E$
parametrized by $\Delta$ and with the base point $0$ if
\begin{enumerate}
\item $h (0, z)=z$ for $z\in E$;
\item for any fixed $c\in
\Delta$, $h (c, \cdot): E\to \hat{\mathbb C}$ is injective;
\item for any fixed $z$, $h (\cdot,z): \Delta \to \hat{\mathbb C}$ is
holomorphic.
\end{enumerate}
\end{definition}

\medskip
For example, for a measurable function $\mu$ on $\hat{\mathbb C}$
with $\|\mu\|_{\infty}=1$,
$$
H(c, z)=H^{c\mu}(z): \Delta\times \hat{\mathbb C}\to \hat{\mathbb C}
$$
is a holomorphic motion of $\hat{\mathbb C}$ parametrized by
$\Delta$ and with the base point $0$.

Note that even continuity does not directly enter into the
definition; the only restriction is in the $c$ direction. However,
continuity is a consequence of the hypotheses from the proof of
the $\lambda$-lemma of Ma\~n\'e, Sad, and Sullivan~\cite[Theorem
2]{ManeSadSullivan}. Moreover, Ma\~n\'e, Sad, and Sullivan proved
in~\cite{ManeSadSullivan} that

\medskip
\begin{lemma}[$\lambda$-Lemma]~\label{ll}
A holomorphic motion of a set $E\subset \hat{\mathbb C}$
parametrized by $\Delta$ and with the base point $0$ can be
extended to a holomorphic motion of the closure of $E$
parametrized by the same $\Delta$ and with the base point $0$.
\end{lemma}

\medskip
Furthermore, Ma\~n\'e, Sad, and Sullivan showed
in~\cite{ManeSadSullivan} that $f(c, \cdot)$ satisfies the Pesin
property. In particular, when the closure of $E$ is a domain, this
property can be described as the quasiconformal property. A
further study of this quasiconformal property was given by
Sullivan and Thurston~\cite{SullivanThurston} and Bers and
Royden~\cite{BersRoyden}. In~\cite{SullivanThurston}, Sullivan and
Thurston proved that there is a universal constant $a>0$ such that
any holomorphic motion of any set $E\subset \hat{\mathbb C}$
parametrized by $\Delta$ and with the basepoint $0$ can be
extended to a holomorphic motion of $\hat{\mathbb C}$ parametrized
by $\Delta_{a}$ and with the base point $0$. In~\cite{BersRoyden},
Bers and Royden showed, by using classical Teichm\"uller theory,
that this constant actually can be taken to be $1/3$. Moreover, in
the same paper, Bers and Royden showed that in any holomorphic
motion $H(c,z): \Delta\times \hat{\mathbb C}\to \hat{\mathbb C}$,
$H(c,\cdot): \hat{\mathbb C}\to \hat{\mathbb C}$ is a
$(1+|c|)/(1-|c|)$-quasiconformal homeomorphism for any fixed $c\in
\Delta$. In both papers~\cite{SullivanThurston,BersRoyden}, they
expected $a=1$. This was eventually proved by Slodkowski
in~\cite{Slodkowski}. Several
people~\cite{Douady,AstalaMartin,Chirka} gave different proofs for
Slodkowski's theorem.

\medskip
\begin{theorem}[Holomorphic Motion Theorem]~\label{hmt}
Suppose
$$
h(c,z): \Delta\times E\to \hat{\mathbb C}
$$
is a holomorphic motion of a set $E\subset \hat{\mathbb C}$
parametrized by $\Delta$ and with the base point $0$. Then $h$ can
be extended to a holomorphic motion
$$
H(c, z): \Delta\times \hat{\mathbb C}\to \hat{\mathbb C}
$$
of $\hat{\mathbb C}$ parametrized by also $\Delta$ and with the
base point $0$. Moreover, for every $c\in \Delta$,
$$
H(c, \cdot): \hat{\mathbb C}\to \hat{\mathbb
C}
$$
is a $(1+|c|)/(1-|c|)$-quasiconformal homeomorphism of
$\hat{\mathbb C}$. The Beltrami coefficient of $H(c, \cdot)$ given
by
$$
\mu(c,z)=\frac{\partial H(c, z)}{\partial
\overline{z}}/\frac{\partial H(c, z)}{\partial z}
$$
is a holomorphic function from $\Delta$ into the unit ball of the
Banach space ${\mathcal L}^{\infty}({\mathbb C})$ of all
essentially bounded measurable functions on ${\mathbb C}$.
\end{theorem}

The reader can read our recent expository
paper~\cite{GardinerJiangWang} for a complete proof of this
theorem and related topics.

\section{Leau-Fatou Flowers and Linearization}

Since the idea in~\cite{Jiang1,Jiang2} plays an important role in
the proof of Theorem~\ref{qcr}, we would like first to use it to
show a conceptual proof of the Fatou linearization theorem. This
proof is again an application of holomorphic motions.

Suppose $f(z)$ is a parabolic germ at $0$. Then there is a
constant $0<r_{0}<1/2$ such that $f(z)$ is conformal with the
Taylor expansion
$$
f(z) = e^{2\pi p i\over q} z+ h.o.t,  \quad (p,q)=1, \quad
|z|<r_{0}.
$$
Suppose $f^{m}\not\equiv id$ for all $m>0$. Then, for appropriate
$r_{0}$,
$$
f^{q}(z) = z (1+ a z^{n} + \epsilon(z)), \quad a\neq 0, \quad
|z|<r_{0},
$$
where $n$ is a multiple of $q$ and $\epsilon(z)$ is given be a
convergent power series of the form
$$\epsilon(z)=a_{n+1}z^{n+1}+a_{n+2} z^{n+2} + \cdots, \quad |z| <r_{0}.$$

Suppose $N\subset \Delta_{r_{0}}$ is a neighborhood of $0$. A
simply connected open set ${\mathcal P}\subset N\cap f^{q}(N)$
with $f^{q}({\mathcal P}) \subset {\mathcal P}$ and $0\in
\overline{\mathcal P}$ is called an attracting petal if $f^{m}(z)$
for $z\in {\mathcal P}$ converges uniformly to $0$ as $m\to
\infty$. An attracting petal ${\mathcal P}'$ for $f^{-1}$ is
called a repelling petal at $0$.

\vspace*{10pt}
\begin{theorem}[The Leau-Fatou flower]~\label{lff}
There exist $n$ attracting petals $\{ {\mathcal
P}_{i}\}_{i=0}^{n-1}$ and $n$ repelling petals $\{ {\mathcal
P}_{j}'\}_{j=0}^{n-1}$ such that
$$
N_{0} =\cup_{i=0}^{n-1}{\mathcal P}_{i}\cup \cup_{j=0}^{n-1}
{\mathcal P}_{j}'
$$
is a neighborhood of $0$.
\end{theorem}

For a proof of this theorem, the reader can refer
to~\cite{MilnorBook}.

For each attracting petal ${\mathcal P}={\mathcal P}_{i}$,
consider the change of coordinate
$$
w= \phi(z) =\frac{d}{z^{n}}, \quad d=-\frac{1}{na},
$$
on ${\mathcal P}$. Suppose the image of ${\mathcal P}$ under
$\phi(z)$ is a right half-plane
$$
R_{\tau} =\{ w\in {\mathbb C}\;\;|\;\; \Re{w} > \tau\}.
$$
Then
$$
z=\phi^{-1} (w) = \sqrt[n]{\frac{d}w}: R_{\tau} \to {\mathcal P}
$$
is a conformal map. The form of $f^{q}$ in the $w$-plane is
$$
F(w) = \phi\circ f\circ \phi^{-1}(w) =w+1 +\eta\Big(
\frac{1}{\sqrt[n]{w}}\Big), \quad \Re{w} >\tau,
$$
where $\eta(\xi)$ is an analytic function in a neighborhood of
$0$. Suppose
$$
\eta(\xi) =b_{1}\xi+b_{2}\xi^{2}+\cdots, \quad |\xi|<r_{1}
$$
is a convergent power series for some $0<r_{1}\leq r_{0}$. Take
$0<r<r_1$ such that
$$
|\eta(\xi)| \leq \frac{1}{2}, \quad \forall \; |\xi| \leq r.
$$
Then $F(R_{\tau}) \subset R_{\tau}$ for any $\tau \geq 1/r^{n}$
since
$$
\Re{F(w)} =\Re{w} +1 +\Re{\eta\Big(\frac{1}{\sqrt[n]{w}}\Big)}
\geq \Re{w} +\frac{1}{2}, \quad \forall \; \Re{w} \geq \tau.
$$

\vspace*{10pt}
\begin{theorem}[Fatou Linearization Theorem]~\label{flt}
Suppose $\tau > 1/r^{n}+1$ is a real number. Then there is a
conformal map $\Psi (w) : R_{\tau} \to \Omega$ such that
$$
F (\Psi (w)) = \Psi (w+1), \quad \forall \; w\in R_{\tau}.
$$
\end{theorem}

Here we give a new proof by involving the holomorphic motions. For
other proofs, see~\cite{MilnorBook}~\cite{CarlesonGamlinBook}.

\subsection{Construction of a holomorphic motion.}
For any $x\geq \tau$, let
$$
E_{0,x} = \{ w\in {\mathbb C}\;|\; \Re{w} =x\}
$$
and
$$
E_{1,x} = \{ w\in {\mathbb C}\;|\; \Re{w} =x+1\}
$$
and let
$$
E_{x}=E_{0,x}\cup E_{1,x}.
$$
Then $E_{x}$ is a subset of $\hat{\mathbb C}$. Define
$$
H_{x} (w) =\left\{
\begin{array}{ll}
        w, & w\in E_{0,x}; \cr
        \Phi(w)=w + \eta\Big(\frac{1}{\sqrt[n]{w-1}}\Big), & w\in E_{1,x}.
\end{array}\right.
$$
Since $H_{x}(w)$ on $E_{0,x}$ and on $E_{1,x}$ are injective,
respectively, and since
$$
\Re(H_{x}(w)) \geq \Re(w) -\frac{1}{2} =x+1-\frac{1}{2} =
x+\frac{1}{2}, \quad w\in E_{1,x},
$$
the images of $E_{0,x}$ and $E_{1,x}$ under $H_{x}(w)$ do not
intersect. So $H_{x}(w)$ is injective. Moreover, $H_{x}(w)$
conjugates $F(w)$ to the linear map $w\mapsto w+1$ on $E_{0,x}$,
that is,
$$
F (H_{x}(w)) =H_{x} (w+1), \quad \forall \; w\in E_{0,x}.
$$

We first introduce a complex parameter $c\in \Delta$ into
$\eta(\xi)$ as follows. Define
$$
\eta (c, \xi) =\eta (cr\xi\sqrt[n]{x-1}) = b_{1}
(cr\xi\sqrt[n]{x-1}) +b_{2} (cr\xi\sqrt[n]{x-1})^{2} +\cdots
$$
for $|c|<1$ and $|\xi| \leq 1/\sqrt[n]{x-1}$. Since
$|cr\xi\sqrt[n]{x-1}|\leq r$, $\eta(c, \xi)$ is a convergent power
series and $|\eta(c, \xi)|\leq 1/2$ for $|c|<1$ and $|\xi|\leq
1/\sqrt[n]{x-1}$. Following this, we therefore introduce a complex
parameter $c\in \Delta$ for $H_{x}(w)$ as defining
$$
H_{x} (c, w) =\left\{
\begin{array}{ll}
        w, & (c, w) \in \Delta \times E_{0,x}; \cr
        \Phi(w)=w + \eta\Big(c, \frac{1}{\sqrt[n]{w-1}}\Big),
        & (c,w)\in \Delta\times E_{1,x}.
\end{array}\right.
$$

\begin{lemma} The map $H_{x}(c,w): \Delta\times \hat{\mathbb C} \to \hat{\mathbb C}$ is
a holomorphic motion.
\end{lemma}

\begin{proof}
(1) It is clear that $H_{x}(0, w) =w$ for $w\in E_{x}$.

(2) Making use of Rouch\'e's theorem, it follows that for any
fixed $c\in \Delta$, $H_{x}(c, \cdot)$ is injective on $E_{0,x}$
and on $E_{1,x}$ , respectively. Since
$$
\Re{H_{x}(c, w)} =\Re{w} +\Re{\eta\Big(c,
\frac{1}{\sqrt[n]{w-1}}\Big)} \geq \Re{w}-\frac{1}{2}, \quad
\forall\; w\in E_{1,x},
$$
the images of $E_{0,x}$ and $E_{1,x}$ under $H_{x}(c,\cdot)$ do
not intersect. So $H_{x}(c, \cdot)$ is injective on $E_{x}$.

(3) For any fixed $w\in E_{0,x}$, $H_{x}(c,w) =w$. So $H(\cdot,w)$
is a holomorphic function of $c$. For any fixed $w\in E_{1,x}$,
$$
H_{x} (c, w) = w +
\eta\Big(\frac{cr\sqrt[n]{x-1}}{\sqrt[n]{w-1}}\Big).
$$
It is a convergent power series of $c$. So it is holomorphic. We
proved the lemma.
\end{proof}

By Theorem~\ref{hmt},
$$
H_{x}(c,w): \Delta\times E_{x}\to \hat{\mathbb C}
$$
can be extended to a holomorphic motion
$$
\widetilde{H}_{x} (c,w): \Delta \times \hat{\mathbb C}\to
\hat{\mathbb C}.
$$
For each $c\in \Delta$,
$$
h_{c} (w) =\widetilde{H}_{x} (c,w): \hat{\mathbb C}\to
\hat{\mathbb C}
$$
is a $(1+|c|)/(1-|c|)$-quasiconformal homeomorphism. When
$c(x)=1/(r\sqrt[n]{x-1})$, $h_{c(x)}$ is a quasiconformal
extension of $H_{x}(w)$ to $\hat{\mathbb C}$ whose quasiconformal
dilatation is less than or equal to
$$
K(x) =
\frac{1+\frac{1}{r\sqrt[n]{x-1}}}{1-\frac{1}{r\sqrt[n]{x-1}}}.
$$
Note that $K(x) \to 1$ as $x\to \infty$.

\subsection{Construction of quasiconformal conjugacies.}
Suppose
$$
S_{x} =\{ w\in {\mathbb C}\;|\; x\leq \Re{w}\leq x+1\}
$$
is the strip bounded by two lines $\Re{w}=x$ and $\Re{w}=x+1$.
Consider the restriction of $h_{c(x)}(w)$ on $S_{x}$ which we
still denote as $h_{c(x)}(w)$.

For any $w_{0}\in R_{\tau}\cup E_{0,\tau}$, let $w_{m}
=F^{m}(w_{0})$. Since $w_{m}-w_{m+1}$ tends to $1$ as $m$ goes to
$\infty$ uniformly on $R_{\tau}\cup E_{0,\tau}$,
$$
\frac{w_{n}-w_{0}}{m} =\frac{1}{m} \sum_{k=1}^{m} (w_{k}-w_{k-1})
\to 1
$$
uniformly on $R_{\tau}\cup E_{0,\tau}$ as $m$ goes to $\infty$. So
$w_{m}$ is asymptotic to $m$ as $m$ goes to $\infty$ uniformly in
any bounded set of $R_{\tau}\cup E_{0,\tau}$.

Let $x_{0}=\tau$ and $x_{m}=\Re(F^{m}(x_{0}))$. Then $x_{m}$ is
asymptotic to $m$ as $m$ goes to $\infty$. For each $m>0$, let
$$
\Upsilon_{m} = F^{-m}(E_{0, x_{m}}\cup\{\infty\}).
$$
It is a curve passing through $x_{0}=\tau$ and $\infty$. Let
$$
\Omega_{m} =F^{-m} (R_{x_{m}}).
$$
It is a domain with the boundary $\Upsilon_{m}$.

Let
$$
S_{i, x_{m}} =F^{-i}(S_{x_{m}}), \quad i=m, m+1, \cdots, 1, 0, -1,
\cdots, -m+1, -m, \cdots.
$$
Then
$$
\Omega_{m} =\cup^{i=m}_{-\infty} S_{i,x_{m}}.
$$

Let
$$
A_{m}= \{ w\in {\mathbb C}\;\;|\;\; \tau+ m\leq \Re{w} \leq
\tau+m+1\}
$$
and let $A_{i, m} =A_{m}-i$ for $i=m, m+1, \cdots, 1, 0, -1,
\cdots, -m+1, -m, \cdots$. Let
$$
\beta_{m} (w) =w+ x_{m} -\tau -m: {\mathbb C}\to {\mathbb C}.
$$
Then it is a conformal map and
$$
\beta_{m} (A_{m}) =S_{x_{m}}.
$$
Define
$$
\psi_{m} (w) = h_{c(x_{m})}\circ \beta_{m} (w).
$$
Then it is a $K(x_{m})$-quasiconformal homeomorphism on $A_{m}$.
Moreover,
$$
F(\psi_{m} (w))=\psi_{m} (w+1), \quad  \forall \; \Re{w}=m+\tau.
$$
Furthermore, define
$$
\psi_{m}(w) = F^{-i}(\psi_{m} (w+i)), \quad \forall \; w\in
A_{i,m}
$$
for $i=m, m-1, \cdots, 1, 0, -1, \cdots, -m+1, -m, \cdots$. Then
it is a $K(x_{m})$-quasiconformal homeomorphism from $R_{\tau}$ to
$\Omega_{m}$ and
$$
F (\psi_{m}(w)) = \psi_{m}(w+1), \quad \forall \; w\in R_{\tau}.
$$

\subsection{Improvement to conformal conjugacy.}

Let $w_{0}=\tau$ and $w_{m} = F^{m}(w_{0})$ for $m=1, 2,\cdots$.
Remember that
$$
R_{x_{m}}=\{ w\in {\mathbb C}\;|\; \Re{w} >x_{m}\}
$$
where $x_{m}=\Re{w_{m}}$.

For any $\tilde{w}_{0}\in R_{x_{m+1}}$, let $\tilde{w}_{m} =
F^{m}(\tilde{w}_{0})$ for $m=1, 2, \cdots$. Since
$$
F'(w) = 1+O\Big(\frac{1}{|w|^{1+\frac{1}{n}}}\Big),\quad w\in
R_{\tau}
$$
and $\tilde{w}_{m}/m \to 1$ as $m\to \infty$ uniformly on any
compact set, there is a constant $C>0$ such that
$$
C^{-1}\leq \frac{|\tilde{w}_{m}-w_{m}|}{|\tilde{w}_{1}-w_{1}|}
=\prod_{k=1}^{m}
\frac{|\tilde{w}_{k+1}-w_{k+1}|}{|\tilde{w}_{k}-w_{k}|} =
\prod_{k=1}^{m} \Big(1+O\Big(\frac{1}{k^{1+\frac{1}{n}}}\Big)\Big)
\leq C
$$
as long as $w_{1}$ and $\tilde{w_{1}}$ keep in a same compact set.
Since
$$
w_{m+1} =w_{m} +1+\eta\Big(\frac{1}{\sqrt[n]{w_{m}}}\Big) \quad
\hbox{and}\quad |\eta\Big(\frac{1}{\sqrt[n]{w_{m}}}\Big)|\leq
\frac{1}{2},
$$
the distance between $w_{m+1}$ and $R_{x_{m}}$ is greater than or
equal to $1/2$. So the disk $\Delta_{1/2}(w_{m+1})$ is contained
in $R_{x_{m}}$. This implies that the disk
$\Delta_{1/(2C)}(w_{1})$ is contained in $\Omega_{m}$ for every
$m=0, 1, \cdots$. Thus the sequence
$$
\psi_{m} (w) : R_{\tau}\to \Omega_{m}, \quad m=1, 2, \cdots
$$
is contained in a weakly compact subset of the space of
quasiconformal mappings. Let
$$
\Psi (w): R_{\tau}\to \Omega
$$
be a limiting mapping of a subsequence. Then $\Psi $ is
$1$-quasiconformal and thus conformal and satisfies
$$
F (\Psi (w)) = \Psi (w+1), \quad \forall w\in R_{\tau}.
$$
This completes the proof of Theorem~\ref{flt}.

\section{Quasiconformal Rigidity}

In this section, we prove Theorem~\ref{qcr} by using a similar
idea in the proof of Theorem~\ref{flt} in \S3.

\subsection{Conformal conjugacies on attracting petals.}

Suppose $f$ and $g$ are two topologically conjugate parabolic
germs. Suppose $f^{m}, g^{m} \not\equiv id$ for all $m>0$. (If
some $f^{m}\equiv \hbox{identity}$, then $g^{m}\equiv
\hbox{identity}$ too.) Suppose $\lambda$ and $n+1$ are their
common multiplier and multiplicity. Suppose $0<r_{0}<1/2$ such
that both $f$ and $g$ are conformal in $\Delta_{r_{0}}$. Without
loss of generality, we assume that $\lambda=1$ and both of $f$ and
$g$ have forms
$$
f(z) = z(1+ z^{n} +o(z^{n})) \quad \hbox{and}\quad g(z) =z(1+
z^{n} +o(z^{n})), \quad |z| <r_{0}.
$$

From Theorem~\ref{lff}, for any small neighborhood $N\subset
\Delta_{r_{0}}$, there are $n$ attracting petals
$\{P_{i,f}\}_{i=0}^{n-1}$ and $n$ repelling petals
$\{P_{i,f}'\}_{i=0}^{n-1}$ for $f$ in $N$. Let us assume that
every $P_{i,f}$ is the maximal attracting petal in $N$. Similarly,
we have the same pattern of attracting petals $\{
P_{i,g}\}_{i=0}^{n-1}$ and the repelling petals $\{
P_{i,g}'\}_{i=0}^{n-1}$ for $g$

From Theorem~\ref{flt} and also~\cite[page 107]{MilnorBook}, for
every $0\leq i\leq n-1$, there is a conformal map
$$
\psi_{i}: P_{i,g}\to P_{i,f}
$$
such that
$$
f(\psi_{i} (z)) = \psi_{i} (g (z)), \quad z\in P_{i,g}.
$$

\subsection{Repelling petals.}

For each $0\leq i\leq n-1$, let
$$
w=\phi (z) =-\frac{1}{nz^{n}}
$$
be the change of coordinate. Then $f$ and $g$ in the
$w$-coordinate system have forms
$$
F(w) =w+1+ \eta_{f} \Big( \frac{1}{\sqrt[n]{w}}\Big) \quad
\hbox{and}\quad G(w) =w+1+ \eta_{g} \Big(
\frac{1}{\sqrt[n]{w}}\Big),
$$
where both of
$$
\eta_{f}(\xi) =a_{1}\xi +a_{2}\xi +\cdots\quad \hbox{and}\quad
\eta_{g}(\xi) =b_{1}\xi +b_{2}\xi +\cdots, \;\; |\xi|<r_{1}
$$
are convergent power series for some number $0<r_{1}<r_{0}$. Take
a number $0<r<r_{1}$ such that
$$
|\eta_{f}(\xi)|, |\eta_{g}(\xi)| \leq \frac{1}{4}, \quad |\xi|
\leq r.
$$
Without loss of generality, we assume that $\eta_{g} (w)\equiv 0$,
that is, $G(w)=w+1$.

Suppose the both repelling petals $P_{i,f}'$ and $P_{i,g}'$ are
changed to a left half-plane
$$
L_{-r^{n}} = \{w\in {\mathbb C} \;\; |\;\; \Re{w} <-r^{n}\}.
$$

\subsection{Construction of a holomorphic motion.}

Take $\tau_{0}=r^{n}$. Let
$$
U_{\tau_{0}} =\{ w\in {\mathbb C} \;\; |\;\;  \Im{w} > \tau_{0}\}
$$
be a upper half-plane and let
$$
D_{-\tau_{0}} =\{ w\in {\mathbb C} \;\; |\;\; \Im{w}<-\tau_{0}\}
$$
be a lower half-plane. Define
$$
\Psi (w) = \left\{ \begin{array}{ll} \phi \circ \psi_{i}\circ
\phi^{-1} (w), & w\in U_{\tau_{0}},\\
\phi\circ \psi_{i+1}\circ \phi^{-1} (w), & w \in D_{\tau_{0}}.
\end{array}\right.
$$
(If $i+1=n$, we consider it as $0$.) Then
$$
F (\Psi (w)) =\Psi (G(w)), \quad w\in U_{\tau_{0}}\cup
D_{\tau_{0}}.
$$
We can have the property that $\Psi (w)/w \to 1$ as $w\to \infty$
(refer to~\cite[pp. 109]{MilnorBook}.

Let $a=e^{-2\pi \tau_{0}}$. Consider the covering map
$$
\xi =\beta (w) = e^{2\pi i w}: {\mathbb C}\to {\mathbb C}\setminus
\{0\}.
$$
Then it maps $U_{\tau_{0}}$ to $\Delta_{a}\setminus \{0\}$ and
$D_{-\tau_{0}}$ to ${\mathbb C}\setminus \overline{\Delta}_{1/a}$.

The inverse of $w=\beta^{-1} (\xi)$ is a multi-valued analytic
function on ${\mathbb C}\setminus \{ 0\}$. We take one branch as
$\beta^{-1}$. Since $\Psi (w)$ is asymptotic to $w$ as $w\to
\infty$, the map
$$
\theta (\xi) =\beta \circ \Psi \circ \beta^{-1} (\xi)
$$
is analytic in $\Delta_{a}$ and in
$\overline{\Delta}_{1/a}^{c}=\hat{\mathbb C}\setminus
\overline{\Delta}_{1/a}$. Suppose
$$
\theta (\xi) =\xi + a_{2} \xi^{2}+\cdots, \quad |\xi| <a
$$
and
$$
\theta (\xi) =\xi + \frac{b_{1}}{ \xi}+\cdots,\quad
|\xi|>\frac{1}{a}
$$
are two convergent power series.

For any $\tau >\tau_{0}$, let $\epsilon = e^{-2\pi \tau}$. Suppose
$\overline{\Delta}_{1/\epsilon^{c}} =(\hat{\mathbb C} \setminus
\overline{\Delta}_{1/\epsilon})$. Let
$$
E=\Delta_{\epsilon}\cup \overline{\Delta}_{1/\epsilon}^{c}.
$$
It is a subset of $\hat{\mathbb C}$. We now introduce a complex
parameter $c\in \Delta$ into $\theta (\xi)$ such that it is a
holomorphic motion of $E$ parametrized by $\Delta$ and with the
base point $0$.

Define
$$
\theta (c, \xi) = \frac{\epsilon}{ca} \theta \Big(
\frac{ca\xi}{\epsilon}\Big)=\xi + a_{2}
\Big(\frac{ca}{\epsilon}\Big)\xi^{2} +\cdots , \quad |c|<1, |\xi|
\leq \epsilon.
$$
and
$$
\theta(c, \xi) = \frac{ca}{\epsilon} \theta \Big( \frac{\epsilon
\xi}{ca} \Big)= \xi + \frac{b_{1}}{\xi}
\Big(\frac{ca}{\epsilon}\Big)^{2} +\cdots , \quad |c|<1, |\xi|
\geq \frac{1}{a}.
$$
We claim that
$$
\theta (c, \xi): \Delta\times E\mapsto \hat{\mathbb C}
$$
is a holomorphic motion. Here is a proof.

(1) It is clear that $\theta(0, \xi)=\xi$ for all $\xi\in E$.

(2) For any fixed $c\neq 0\in \Delta$, $\theta(c,\xi)$ on
$\Delta_{\epsilon}$ is a conjugation map of $\theta (\xi)$ by the
linear map $\xi \mapsto (ca/\epsilon)\xi$. And $\theta(c,\xi)$ on
$\overline{\Delta}_{1/\epsilon}^{c}$ is a conjugation map of
$\theta (\xi)$ by the linear map $\xi \mapsto (\epsilon/(ca))\xi$.
So they are injective. Since the image $\theta (c,
\Delta_{\epsilon})$ is contained in $\Delta_{a}$ and the image
$\theta (c, \overline{\Delta}_{1/\epsilon}^{c})$ is contained in
$\overline{\Delta}_{1/a}^{c}$, they do not intersect. So $\theta
(c, \cdot)$ on $E$ is injective.

(3) For any fixed $\xi\in \Delta_{\epsilon}$, since
$|ca\xi/\epsilon| < a$ for $|c|<1$, it is a convergent power
series of $c$. So $\theta (\cdot, \xi)$ is holomorphic on $c$. For
any fixed $\xi\in \overline{\Delta}_{1/\epsilon}^{c}$, since
$|\epsilon\xi/(ca)|
>1/a$ for $|c|<1$, so it is a convergent power
series of $c$. So $\theta(\cdot, \xi)$ is holomorphic on $c$. We
proved the claim.

Let $E_{\tau} = U_{\tau}\cup D_{-\tau}$. Then $\beta (E_{\tau})
=E$. Thus we can lift the holomorphic motion
$$
\theta (c, \xi): \Delta\times E\to \hat{\mathbb C}
$$
to get a holomorphic motion
$$
h_{0}(c, w): \Delta\times E_{\tau}\to \hat{\mathbb C}.
$$
When $c(\tau) =\epsilon/a$, $h(c(\tau), w) =\Psi (w)$.

Let $w_{1} = -\tau +i\tau$ and $w_{2}= -\tau-i\tau$. Consider the
vertical segment connecting them
$$
s_{\tau} =\{ tw_{1} +(1-t) w_{2}\;|\; 0\leq t\leq 1\}.
$$
Let
$$
s_{\tau}' = s_{\tau}+1 =\{ tw_{1} +(1-t) w_{2} +1\;|\; 0\leq t\leq
1\}.
$$
Define
$$
h_{1}(c, tw_{1}+(1-t)w_{2}) = t h_{0}(c, w_{1}) +(1-t) h_{0}(c,
w_{2}) : \Delta\times s_{\tau}\to \hat{\mathbb C}
$$
and
$$
h_{2} (c, tw_{1}+(1-t)w_{2}+1) = F(h_{1}(c, tw_{1}+(1-t)w_{2}):
\Delta\times s_{\tau}'\to \hat{\mathbb C}.
$$
Both of these maps are holomorphic motions. Since
$$
h_{2} (c, tw_{1}+(1-t)w_{2}+1)
$$
$$
=t h_{0}(c, w_{1}) +(1-t) h_{0}(c, w_{2})+1 +\eta (t h_{0}(c,
w_{1}) +(1-t) h_{0}(c, w_{2}))
$$
and since
$$
|\eta(w)|\leq 1/4, \quad  \forall \; |w|\geq \tau,
$$
the images of these two holomorphic motions do not intersect.
Therefore, we define a holomorphic motion
$$
h(c,w) =\left\{ \begin{array}{ll}
                 h_{0}(c,w), & (c,w)\in \Delta\times E_{\tau};\\
                 h_{1}(c,w), & (c,w) \in \Delta\times s_{\tau};\\
                 h_{2}(c,w), & (c,w) \in \Delta\times s_{\tau}'
       \end{array}\right.
$$
of $\Sigma =E_{\tau}\cup s_{\tau}\cup s_{\tau}'$ parametrized by
$\Delta$ and with base point $0$.

For $c(\tau)=\epsilon/a$, $h(c(\tau),w)$ is a conjugacy from $F$
to $G$ on $E_{\tau}\cup s_{\tau}$. i.e.,
$$
F(h(c(\tau), w)) = h(c(\tau), G(w)), \quad w\in E_{\tau}\cup
s_{\tau}.
$$

By Theorem~\ref{hmt}, there is an extended holomorphic motion of
$h(c,w)$,
$$
H(c,w): \Delta\times \hat{\mathbb C}\mapsto \hat{\mathbb C}.
$$
such that for each $c\in \Delta$, $H(c,\cdot)$ is a
$(1+|c|)/(1-|c|)$-quasiconformal homeomorphism of $\hat{\mathbb
C}$.

Let $H(w) = H(c(\tau), w)$ and
$$
K(\tau) =\frac{1+c(\tau)}{1-c(\tau)}.
$$
Note that $K(\tau) \to 1$ as $\tau\to \infty$. Then $H(w)$ is a
$K(\tau)$-quasiconformal homeomorphism of $\hat{\mathbb C}$ such
that
$$
H(w) = h(c(\tau), w), \quad \forall w\in \Sigma.
$$

Let
$$
A_{0} =\{ w\in {\mathbb C}\;|\; -\tau\leq \Re{w}\leq -\tau +1\}
$$
and $A_{-m} =A_{0}-m$ for $m=1, 2, \cdots$. Define
$$
\Psi (w) = F^{-m}(H(w+m)), \quad w\in A_{-m}, \quad m=0, 1,
\cdots.
$$
Then $\Psi (w)$ is a $K(\tau)$-quasiconformal homeomorphism
defined on the left half-plane
$$
L_{-\tau+1} =\{ w \in {\mathbb C}\;|\; \Re{w} \leq -\tau+1\}
$$
and extends $\Psi (w)$ on $U_{\tau}\cup D_{-\tau}$.

Now let
$$
\psi (z) = \phi^{-1} \circ \Psi \circ \phi (z).
$$
It extends
$$
\psi_{i}: P_{i,g}\to P_{i,f}\quad \hbox{and}\quad \psi_{i+1}:
P_{i+1,g}\to P_{i+1,f}
$$
in a small neighborhood $N$ to a $K(\tau)$-quasiconformal
homeomorphism
$$
\psi (z) : P_{i,g}\cup P_{i,g}'\cup P_{i+1,g}\to P_{i,f}\cup
P_{i,f}'\cup P_{i+1,f}
$$
and
$$
f\circ \psi (z) =\psi \circ g(z), \quad \forall z \in P_{i,g}\cup
P_{i,g}'\cup P_{i+1,g}.
$$

If we work out the above for every $0\leq i\leq n-1$, we get that
for any $\varepsilon >0$, there is a neighborhood
$U_{\varepsilon}$ of $0$ and a
$(1+\varepsilon)/(1-\varepsilon)$-quasiconformal homeomorphism
$$
\psi (z): U_{\varepsilon}  \to V_{\varepsilon} =\psi
(V_{\varepsilon})
$$
such that it extends every $\psi_{i}: P_{i,g}\to P_{i,f}$ in
$U_{\varepsilon}$ and such that
$$
f\circ \psi (z) =\psi \circ g(z), \quad \forall z\in
U_{\varepsilon}.
$$
This gives a complete proof of Theorem~\ref{qcr}.

\section{Gluing Germs in the Riemann Sphere}

As a by-product of our proofs of Theorems~\ref{qcr} and~\ref{flt},
we prove a theorem saying that we can use a quasiconformal
homeomorphism to glue an arbitrary finite number of parabolic
germs in the Riemann sphere at different points. This theorem may
be thought of as a generalization of Ahlfors-Weill's extension
Theorem which basically consider one germ. Again our proof is
based on holomorphic motions. Suppose $\Delta_{r}(z)$ is the disk
of radius $r>0$ centered at $z$.

\vspace*{10pt}
\begin{theorem}~\label{conn}
Suppose $\{ f_{i}\}_{i=1}^{k}$ is a finite number of parabolic
germs at distinct points $\{z_{i}\}_{i=1}^{k}$ in the complex
plane ${\mathbb C}$. Then for every $\epsilon>0$ there exists a
number $r>0$ and a $(1+\epsilon)$-quasiconformal homeomorphism $f$
of $\hat{\mathbb{C}}$ such that
$$
f|_{\Delta_{r}(z_i)}=f_i|_{\Delta_{r}(z_i)}, \quad i=1, \cdots, k.
$$
\end{theorem}

\begin{proof}

Denote
$$
B_{i}(r) = f_{i}(\Delta_{r}(z_{i})).
$$
Let $r_{0}>0$ be a number such that
$$
B_{i}(r) \cap B_{j}(r)=\emptyset, \quad 1\leq i\not= j\leq k,
\quad 0<r\leq r_{0}.
$$
Let
$$
E_{r}=\cup_{i=1}^{n} \overline{\Delta}_{r}(z_{i})
$$
be a closed subset of $\hat{\mathbb C}$.

For any $0<r\leq r_{0}$, write
$$
f_{i}(z) = z + a_{i,2}(z-z_{i})^{2} +\cdots + a_{i,n}
(z-z_{i})^{n} +\cdots, \quad |z-z_{i}|\leq r.
$$
Let
$$
\eta_{i}(\xi)=  a_{i,2}\xi^{2} +\cdots + a_{i,n} \xi^{n} +\cdots.
$$
Then
$$
f_{i}(z) =z+ \eta_{i}(z-z_{i}), \quad |z-z_{i}|\leq r.
$$
Let $\phi(z)$ be defined on $E_{r}$ as
$$
\phi (z) = f_{i}(z) =z +\eta_{i}(z-z_{i}) \quad \hbox{for} \quad
|z-z_{i}|\leq r,  \quad \quad i=1, \cdots, k.
$$

We introduce a complex parameter $c\in \Delta$ into $\phi(z)$ as
follows. Define
$$
\phi (c, z) = z + \frac{r}{cr_{0}} \eta_{i} \Big( \frac{cr_{0}}{r}
(z-z_{i})\Big), \quad |z-z_{i}|\leq r, \quad i=1, \cdots, k.
$$
Then
$$
\phi (c, z): \Delta\times E_{r}\to \hat{\mathbb C}
$$
is a map. We will check that it is a holomorphic motion.

For any fixed $c\in \Delta$, we have
$$
\phi_{z}' (c,z) =1+ \eta_{i}' \Big( \frac{cr_{0}}{r}
(z-z_{i})\Big) , \quad |z-z_{i}|\leq r, \quad i=1, \cdots, k.
$$
By picking $r_{0}>0$ small enough, we can assume
$$
|f_{i}'(z) | = |1+\eta_{i}' (z-z_{i})| \geq 1- |\eta_{i}'
(z-z_{i})|>0, \quad |z-z_{i}|< r_{0}, \quad i=1, \cdots, k.
$$
Thus
$$
\phi_{z}'(c,z) \not= 0, \quad \forall |z-z_{i}|\leq r, \quad i=1,
\cdots, k.
$$
We get that $\phi (c, z)$ on each $\Delta_{r}(z_{i})$ is
injective. But images of $\Delta_{r}(z_{i})$ and
$\Delta_{r}(z_{j})$, for $1\leq i\neq j\leq k$, under $\phi (c,z)$
are pairwise disjoint. So $\phi (c, z)$ is injective on $E_{r}$.

It is clear that
$$
\phi (0, z)=z, \quad z\in E_{r}.
$$

For any fixed $z\in \Delta_{r}(z_{i})$, $1\leq i\leq k$,
$$
\phi (c,z) =z + \frac{r}{cr_{0}} \eta_{i} \Big( \frac{cr_{0}}{r}
(z-z_{i})\Big).
$$
Since
$$
\Big| \frac{cr_{0}}{r} (z-z_{i})\Big| <r_{0},
$$
$$
\eta_{i} \Big( \frac{cr_{0}}{r} (z-z_{i})\Big)
$$
is a convergent power series of $c\neq 0\in \Delta$. For $c=0$,
$\phi(0,z)=z$. So $\phi (c, z)$ is holomorphic with respect to
$c\in \Delta$.

Therefore,
$$
\phi (c,z) : \Delta\times E(r) \to \hat{\mathbb C}
$$
is a holomorphic motion.

Following Theorem~\ref{hmt}, we have an extended holomorphic
motion
$$
\tilde{\phi}(c,z): \Delta\times \hat{\mathbb C};
$$
that is, $\widetilde{\phi}(c,z)|\Delta\times E_r=\phi (c,z)$.
Moreover, for any $c\in \Delta$, $\widetilde{\phi}(c,\cdot)$ is a
$(1+|c|)/(1-|c|)$-quasiconformal mapping.

Let
$$
f(z) = \widetilde{\phi}\Big( \frac{r}{r_{0}}, z\Big).
$$
Then $f(z)$ is a $(1+r/r_{0})/(1-r/r_{0})$-quasiconformal
homeomorphism. Furthermore,
$$
f|\Delta_{r}(z_{i}) = \widetilde{\phi} \Big(
\frac{r}{r_{0}},z\Big) |\Delta_{r} (z_{i})= \phi
\Big(\frac{r}{r_{0}},z\Big) |\Delta_{r} (z_{i})=f_{i}|\Delta_{r}
(z_{i}).
$$

Thus for any given $\epsilon>0$, we take $r= (2\epsilon r_{0})
/(1+\epsilon)$; then $f$ is a $(1+\epsilon)$-quasiconformal
mapping and extends $f_{i}$ for all $i=1, 2, \cdots, k$. We
completed the proof.
\end{proof}

\vspace*{10pt}
\begin{corollary}~\label{connmore}
Suppose $\{ f_{i}\}_{i=1}^{k}$ is a finite number of germs at
distinct points $\{z_{i}\}_{i=1}^{k}$ such that
$\lambda_{i}=f_{i}'(z_{i}) \neq 0$ for $1\leq i\leq k$. Then for
every $\epsilon>0$ there exist a number $s>0$ and a
$(1+\epsilon)$-quasiconformal homeomorphism $f$ of
$\hat{\mathbb{C}}$ such that
$$
f|_{\Delta_{s}(z_i)}=f_i|_{\Delta_{s}(z_i)}, \quad i=1, \cdots, k.
$$
\end{corollary}

\begin{proof}
First suppose $r_{0}>0$ and suppose that
$$
f_{i} (z) = z_{i}+\lambda_{i}(z-z_{i}),\quad z\in D_{r_{0}}
(z_{i}), \quad \lambda_{i}\neq 0, \quad 1\leq i\leq k.
$$
Suppose
$$
\overline{\Delta}_{r_{0}}(z_{i})\cap
\overline{\Delta}_{r_{0}}(z_{j})=\emptyset,\quad  \hbox{for all
$0\leq i\neq j\leq k$}.
$$

Let
$$
a=\max \{ |\log \lambda_{i}|\;\;|\;\; 1\leq i\leq k \},
$$
and let
$$
s= r_{0}e^{-\frac{a}{r}}
$$
for any $0<r< r_{0}$.

Let $\Delta_{s} (z_{i})$ and $E_{s} =\cup_{i=1}^{k}
\overline{\Delta}_{s} (z_{i})$. Define
$$
\phi (c, z) = z_{i} + e^{\frac{c}{r} \log \lambda_{i}} (z-z_{i}),
\quad c\in \Delta, \;\; z\in \overline{\Delta}_{s}(z_{i}).
$$
We will check that
$$
\phi (c, z): \Delta\times E_{s} \to \hat{\mathbb C}
$$
is a holomorphic motion.

For $c=0$, we have $\phi (0, z) =z$ for all $z\in E_{s}$.

For each fixed $c\in \Delta$, $\phi (c, z)$ on each
$\overline{\Delta}_{s} (z_{i})$ is injective, but the image of
$\Delta_{s}(z_{i})$ under $\phi (c, z)$ is contained in
$\Delta_{r_{0}}(z_{i})$. So $\phi(c,z)$ on $E_{s}$ is injective.

For fixed $z\in E_{s}$, it is clear that $\phi (c,z)$ is
holomorphic with respect to $c\in \Delta$.

So
$$
\phi (c, z): \Delta\times E_{s} \to \hat{\mathbb C}\to
\hat{\mathbb C}
$$
is a holomorphic motion.

By Theorem~\ref{hmt}, we have an extended holomorphic motion
$$
\widetilde{\phi}(c,z): \Delta\times \hat{\mathbb C}\to
\hat{\mathbb C};
$$
that is, $\widetilde{\phi}(c,z)|\Delta\times E_s=\phi (c,z)$.
Moreover, for any $c\in \Delta$, $\widetilde{\phi}(c,\cdot)$ is a
$(1+|c|)/(1-|c|)$-quasiconformal homeomorphism.

Let $f(z) = \widetilde{\phi}(r, z)$. Then $f(z)$ is a
$(1+r)/(1-r)$-quasiconformal homeomorphism. Furthermore,
$$
f|\Delta_{s}(z_{i}) = \widetilde{\phi} (r,z) |\Delta_{s} (z_{i})=
\phi (r,z) |\Delta_{s} (z_{i})=f_{i}|\Delta_{s} (z_{i}).
$$

Now we consider the general situation,
$$
f_{i} (z) = z_{i}+\lambda_{i}(z-z_{i}) +a_{2,i}(z-z_{i})^{2}
+\cdots,\quad z\in \Delta_{r_{0}} (z_{i}), \quad \lambda_{i}\neq
0,\;\; 1\leq i\leq k.
$$
Let
$$
g_{i} (z) =z_{i} +\lambda^{-1}_{i} (z-z_{i}), \quad 1\leq i\leq k.
$$
Then
$$
F_{i}(z) =f_{i}\circ g_{i} (z) = z+
\frac{a_{2,i}}{\lambda_{i}^{2}} (z-z_{i})^{2} +\cdots, \quad 1\leq
i\leq k,
$$
are all parabolic germs.

From Theorem~\ref{conn} and the above argument, for any $\epsilon
>0$, we have $0<s<r\leq r_{0}$ and two $\sqrt{1+\epsilon}$-quasiconformal
homeomorphisms $F(z)$ and $G(z)$ of $\hat{\mathbb C}$ such that
$$
F|\Delta_{r}(z_{i}) =F_{i}|\Delta_{r}(z_{i})\quad \hbox{and}\quad
G|\Delta_{s}(z_{i}) =g_{i}^{-1}|\Delta_{s}(z_{i})
$$
and such that
$$
G(\Delta_{s}(z_{i})) \subset \Delta_{r}(z_{i}).
$$
Then $f(z) =F\circ G (z)$ is a $(1+\epsilon)$-quasiconformal
homeomorphism of $\hat{\mathbb C}$ such that
$$
f|\Delta_{s}(z_{i}) = F\circ G |\Delta_{s}(z_{i}) = f_{i}\circ
g_{i}\circ g_{i}^{-1} |\Delta_{s}(z_{i}) =f_{i}|\Delta_{s}(z_{i}).
$$
We completed the proof.
\end{proof}

\bibliographystyle{amsalpha}

\end{document}